\documentclass[reqno,11pt]{amsart}
\usepackage{bm,url,amsmath,amssymb,amsthm}

\usepackage{graphicx}
\usepackage{epic}
\usepackage{eepic}

\usepackage{enumerate}
\usepackage[top=1in, bottom=1in, left=1in, right=1in]{geometry}
\usepackage{setspace}
\usepackage{mathrsfs}

\allowdisplaybreaks

\newcommand{\R}{\mathbb{R}}

\newcommand{\N}{\mathbb{N}}
\newcommand{\rar}{\rightarrow}
\newcommand{\bsl}{\backslash}

\newtheorem{theorem}{Theorem}

\newtheorem{lemma}{Lemma}

\theoremstyle{definition}
\newtheorem{remark}{Remark}

\title[Preimages of iterates of the Farey map]{An effective estimate for the Lebesgue measure of preimages of iterates of the Farey map}
\author{Byron Heersink}

\address{Department of Mathematics, University of Illinois at Urbana-Champaign, Urbana, IL 61801}

\email{heersin2@illinois.edu}

\begin{document}

\begin{abstract}
Using techniques from infinite ergodic theory, Kesseb\"ohmer and Stratmann determined the asymptotic behavior of the Lebesgue measure of sets of the form $F^{-n}[\alpha,\beta]$, where $[\alpha,\beta]\subseteq(0,1]$ and $F$ is the Farey map. In this paper, we provide an effective version of this result, employing mostly basic properties of the transfer operator of the Farey map and an application of Freud's effective version of Karamata's Tauberian theorem.

\ 

\end{abstract}

\maketitle

\section{Introduction}

In a letter to Laplace in 1812, Gauss posed the problem of estimating the error 
\[\lambda([a_1,a_2,\ldots]:[a_{n+1},a_{n+2},\ldots]<u)-\frac{\log(u+1)}{\log2}\]
as $n$ approaches $\infty$, where $\lambda$ is the Lebesgue measure, $u\in(0,1)$ is fixed, and
\[[a_1,a_2,\ldots]:=\cfrac{1}{a_1+\cfrac{1}{a_2+\cdots}}\qquad(a_i\in\N)\]
denotes a regular continued fraction expansion. Let $G:[0,1]\rar[0,1]$ be the Gauss map defined by
\[G(x):=\begin{cases}\{1/x\}&\mbox{if }x\neq0\\ 0&\mbox{if }x=0\end{cases}\]
and $d\nu=d\lambda/((1+x)\log2)$ be the Gauss invariant measure. This problem is equivalent to estimating
\begin{equation}\label{e01}
\lambda(G^{-n}[0,u))-\nu[0,u).\qquad(n\rar\infty)
\end{equation}

Throughout this paper, we write $f(x)=O(g(x))$, or equivalently $f(x)\ll g(x)$, as $x\rar\infty$ if there exist constants $M,N>0$ such that $|f(x)|\leq M|g(x)|$ for all $x\geq N$ (when $f(x)=O_{a_1,\ldots,a_m}(g(x))$, the constants $M$ and $N$ depend on $a_1,\ldots,a_m$); and $f(x)\sim g(x)$ as $x\rar\infty$ if $\lim_{x\rar\infty}f(x)/g(x)=1$.

L\'evy first showed that \eqref{e01} is $O(q^n)$ for $q=3.5-2\sqrt{2}$, and Wirsing determined the optimal value of $q$ as $0.30366\ldots$ by discovering the spectral gap in the transfer operator of the Gauss map. An exact solution to Gauss's problem was first given by Babenko, who proved that the transfer operator is compact when restricted to a certain Hilbert space of functions. This result was later extended by Mayer and Roepstorff. (See \cite{IK} for history and details.)

This paper is concerned with the analogue of Gauss's problem for the Farey map $F:[0,1]\rar[0,1]$ defined by
\[F(x):=\begin{cases}x/(1-x)&\mbox{if }0\leq x\leq\frac{1}{2}\\
(1-x)/x&\mbox{if }\frac{1}{2}<x\leq1.
\end{cases}\]
Specifically, we analyze the asymptotic behavior of $\lambda(F^{-n}[u,1])$. The Gauss map is conjugate to the induced transformation of the Farey map on $[1/2,1]$ (see \cite[Sections 3 and 8]{I1}). In spite of this relationship, the Gauss and Farey maps exhibit very different behavior, one of the reasons being that $F$ preserves the infinite measure $d\mu=d\lambda/x$.

In the special case $u=1/2$, $F^{-(n-1)}[u,1]$ is the $n$th sum-level set for continued fractions
\[\mathscr{C}_n:=\{[a_1,a_2,\ldots]\in[0,1]:\sum_{i=1}^ka_i=n\mbox{ for some }k\in\N\}.\]
This follows from the fact that $F$ maps continued fractions $[a_1,a_2,\ldots]$ as follows:
\[F([a_1,a_2,\ldots])=\begin{cases}
[a_1-1,a_2,\ldots]&\mbox{if }a_1\geq2\\
[a_2,a_3,\ldots]&\mbox{if }a_1=1,
\end{cases}\]
from which it is straighforward to see that $F^{-1}(\mathscr{C}_n)=\mathscr{C}_{n+1}$, and hence $\mathscr{C}_n=F^{-(n-1)}(\mathscr{C}_1)=F^{-(n-1)}[1/2,1]$ (see \cite[Lemma 2.1]{KSt2}).

\begin{figure}
\centering
\includegraphics[width=0.5\textwidth]{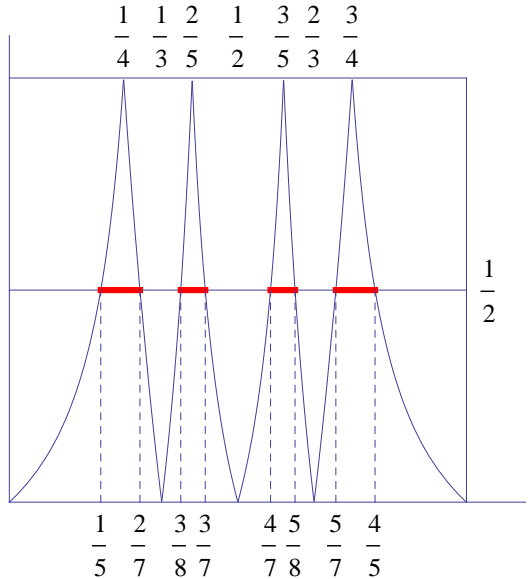}
\caption{The graph of $F^3$ and $\mathscr{C}_4$ shown as the inverse image $F^{-3}[1/2,1]$.}
\label{fig:1}
\end{figure}

Confirming a conjecture of Fiala and Kleban \cite{FK}, Kesseb\"ohmer and Stratmann proved the asymptotic equivalence \cite[Theorem 1.3]{KSt2}
\[\lambda(\mathscr{C}_n)\sim\frac{1}{\log_2n},\qquad(n\rar\infty)\]
and later, they generalized this result by proving that
\begin{equation}\label{e02}
\lambda(F^{-n}[u,1])\sim\frac{\log(1/u)}{\log n}\qquad(n\rar\infty)
\end{equation}
for all $u\in(0,1]$. This in fact follows from their stronger result \cite[Proposition 1.3]{KSt1} that the sets of the form $F^{-n}[\alpha,\beta]$, with $[\alpha,\beta]\subseteq(0,1]$, equidistribute in $[0,1]$. Their proofs applied deep results in infinite ergodic theory following from the work of Aaronson \cite{A1, A2}, and Kesseb\"ohmer and Slassi \cite{KSl1, KSl2}, to the Farey map. In particular, they used the inverse relationship between the wandering rate of a uniformly returning set of the Farey map and the decay of the iterates of the Farey map's transfer operator. The goal of this paper is to prove the following result, which provides an effective version of \eqref{e02}.

\begin{theorem}\label{T1}
For any interval $[\alpha,\beta]\subseteq(0,1]$, we have
\begin{equation}\label{e03}
\lambda(F^{-(n-1)}[\alpha,\beta])=\frac{\log(\beta/\alpha)}{\log n}\left(1+O_{\alpha,\beta}\left(\frac{1}{\log n}\right)\right).\qquad(n\rar\infty)
\end{equation}
\end{theorem}

Instead of proving Theorem \ref{T1} directly, we prove the following result, part \eqref{e05} of which resembles \cite[Theorem 1.2]{KSt2}.

\begin{theorem}\label{T2}
Let $u\in(0,1)$ and $\mathscr{C}_n^u:=F^{-(n-1)}[u,1]$. Then:
\begin{enumerate}[(a)]
\item $(\lambda(\mathscr{C}_n^u))_n$ is a decreasing sequence;\label{e04}

\item $\displaystyle
\sum_{k=0}^n\lambda(\mathscr{C}_{k+1}^u)=\frac{n\log(1/u)}{\log n}\left(1+O_u\left(\frac{1}{\log n}\right)\right).\qquad(n\rar\infty)$\label{e05}
\end{enumerate}
\end{theorem}

To establish Theorem \ref{T1} from this, note that \eqref{e04} and \eqref{e05} imply that for all $u\in(0,1)$,
\begin{align*}
\lambda(\mathscr{C}_n^u)&\leq\frac{1}{n}\sum_{k=0}^n\lambda(\mathscr{C}_{k+1}^u)=\frac{\log(1/u)}{\log n}\left(1+O_u\left(\frac{1}{\log n}\right)\right)\mbox{, and}\\
\lambda(\mathscr{C}_n^u)&\geq\frac{1}{n}\sum_{k=n+1}^{2n}\lambda(\mathscr{C}_{k+1}^u)=\frac{1}{n}\left(\frac{2n\log(1/u)}{\log2n}-\frac{n\log(1/u)}{\log n}\right)\left(1+O_u\left(\frac{1}{\log n}\right)\right)\\
&=\left(\frac{\log(1/u)}{\log n}+O_u\left(\frac{1}{\log^2n}\right)\right)\left(1+O_u\left(\frac{1}{\log n}\right)\right)=\frac{\log(1/u)}{\log n}\left(1+O_u\left(\frac{1}{\log n}\right)\right),
\end{align*}
and hence
\[\lambda(\mathscr{C}_n^u)=\frac{\log(1/u)}{\log n}\left(1+O_u\left(\frac{1}{\log n}\right)\right).\qquad(n\rar\infty)\]
Then subtracting this expression for $u=\beta$ from that for $u=\alpha$ yields \eqref{e03}.

To prove Theorem \ref{T2}, we make use of the transfer operator $\hat{F}:L^1(\mu)\rar L^1(\mu)$ of $F$, which is the positive linear operator satisfying
\[\int_{B}\hat{F}f\,d\mu=\int_{F^{-1}(B)}f\,d\mu\mbox{, for all Borel subsets }B\subseteq[0,1]\mbox{ and }f\in L^1(\mu),\]
and is given by
\begin{equation}\label{e06}
\hat{F}f(x)=\frac{f(x/(1+x))+xf(1/(1+x))}{1+x}.
\end{equation}
Unlike how the Gauss problem was addressed, we do not appeal to any spectral properties of $\hat{F}$, as obtaining an effective version of \eqref{e02} through such means appears to be difficult, by the examinations of the spectrum by Isola \cite{I2} and Prellberg \cite{P}. We also forego applying strong, general results from infinite ergodic theory to $\hat{F}$. Instead, we establish estimates involving sums of the iterates of $\hat{F}$ specifically, and make careful applications of the equality \eqref{e08} following from \cite[Lemma 3.8.4]{A1} and Karamata's Tauberian theorem \cite{K}, which are important results underlying much of the machinery used in \cite{KSt1} and \cite{KSt2}, so as to obtain error terms. In particular, we make an application of Freud's effective version of Karamata's theorem \cite{F} in establishing an asymptotic estimate of a certain weighted sum of the measures $\lambda(\mathscr{C}_{k+1}^u)$ from an estimate of its Laplace transform derived from \eqref{e08}. We can then remove the weights to prove \eqref{e05} by a standard analytic number theory argument. See \cite{MT} for other asymptotic results derived from operator renewal theory involving the iterates of transfer operators of infinite measure-preserving systems.

\begin{figure}[t]
\begin{center}
\unitlength 0.43mm
\begin{picture}(0,105)(0,-5)

\dottedline{2}(-120,80)(-120,100)(0,80)(120,100)(120,80)
\dottedline{2}(-120,80)(-120,60)
\dottedline{2}(-120,60)(-120,80)(-60,60)(0,80)(60,60)(120,80)(120,60)
\dottedline{2}(0,80)(0,60)

\dottedline{2}(-120,40)(-120,60)(-90,40)(-60,60)(-30,40)(0,60)(30,40)(60,60)(90,40)(120,60)(120,40)
\dottedline{2}(-60,60)(-60,40) \dottedline{2}(0,60)(0,40)
\dottedline{2}(60,60)(60,40)
\dottedline{2}(-120,20)(-120,40)(-105,20)(-90,40)(-75,20)(-60,40)(-45,20)(-30,40)(-15,20)(0,40)(15,20)
(30,40)(45,20)(60,40)(75,20)(90,40)(105,20)(120,40)(120,20)
\dottedline{2}(-90,40)(-90,20) \dottedline{2}(-60,40)(-60,20)
\dottedline{2}(-30,40)(-30,20)
\dottedline{2}(0,40)(0,20) \dottedline{2}(30,40)(30,20)
\dottedline{2}(60,40)(60,20)
\dottedline{2}(90,40)(90,20)

\dottedline{2}(-120,5)(-120,20)(-112.5,5)(-105,20)(-97.5,5)(-90,20)(-82.5,5)(-75,20)(-67.5,5)(-60,20)
(-52.5,5)(-45,20)(-37.5,5)(-30,20)(-22.5,5)(-15,20)(-7.5,5)(0,20)
\dottedline{2}(120,5)(120,20)(112.5,5)(105,20)(97.5,5)(90,20)(82.5,5)(75,20)(67.5,5)(60,20)
(52.5,5)(45,20)(37.5,5)(30,20)(22.5,5)(15,20)(7.5,5)(0,20)

\dottedline{2}(-105,20)(-105,5) \dottedline{2}(-90,20)(-90,5)
\dottedline{2}(-75,20)(-75,5)
\dottedline{2}(-60,20)(-60,5) \dottedline{2}(-45,20)(-45,5)
\dottedline{2}(-30,20)(-30,5)
\dottedline{2}(-15,20)(-15,5) \dottedline{2}(0,20)(0,5)
\dottedline{2}(105,20)(105,5)
\dottedline{2}(90,20)(90,5) \dottedline{2}(75,20)(75,5)
\dottedline{2}(60,20)(60,5)
\dottedline{2}(45,20)(45,5) \dottedline{2}(30,20)(30,5)
\dottedline{2}(15,20)(15,5)

\Thicklines
\path(0,80)(120,80)
\path(-60,60)(60,60)
\path(-90,40)(-30,40) \path(30,40)(90,40)
\path(-105,20)(-75,20) \path(-45,20)(-15,20) \path(105,20)(75,20)
\path(45,20)(15,20)
\path(-112.5,5)(-97.5,5) \path(-82.5,5)(-67.5,5) \path(-52.5,5)(-37.5,5)
\path(-22.5,5)(-7.5,5)
\path(112.5,5)(97.5,5) \path(82.5,5)(67.5,5) \path(52.5,5)(37.5,5)
\path(22.5,5)(7.5,5)

\put(-125,100){\makebox(0,0){{\small $\frac{0}{1}$}}}
\put(0,88){\makebox(0,0){{\small $\frac{1}{2}$}}}
\put(135,80){\makebox(0,0){{\small ${\mathscr C}_1$}}}
\put(125,100){\makebox(0,0){{\small $\frac{1}{1}$}}}
\put(-125,80){\makebox(0,0){{\small $\frac{0}{1}$}}}
\put(125,80){\makebox(0,0){{\small $\frac{1}{1}$}}}
\put(-125,60){\makebox(0,0){{\small $\frac{0}{1}$}}}
\put(-60,67){\makebox(0,0){{\small $\frac{1}{3}$}}}
\put(60,67){\makebox(0,0){{\small $\frac{2}{3}$}}}
\put(125,60){\makebox(0,0){{\small $\frac{1}{1}$}}}
\put(-125,60){\makebox(0,0){{\small $\frac{0}{1}$}}}
\put(135,60){\makebox(0,0){{\small ${\mathscr C}_2$}}}

\put(-125,40){\makebox(0,0){{\small $\frac{0}{1}$}}}
\put(-90,47){\makebox(0,0){{\small $\frac{1}{4}$}}}
\put(-30,47){\makebox(0,0){{\small $\frac{2}{5}$}}}
\put(30,47){\makebox(0,0){{\small $\frac{3}{5}$}}}
\put(90,47){\makebox(0,0){{\small $\frac{3}{4}$}}}
\put(125,40){\makebox(0,0){{\small $\frac{1}{1}$}}}
\put(135,40){\makebox(0,0){{\small ${\mathscr C}_3$}}}
\put(135,20){\makebox(0,0){{\small ${\mathscr C}_4$}}}
\put(135,5){\makebox(0,0){{\small ${\mathscr C}_5$}}}

\put(-120,100){\makebox(0,0){{\tiny $\bullet$}}}
\put(120,100){\makebox(0,0){{\tiny $\bullet$}}}
\put(-120,80){\makebox(0,0){{\tiny $\bullet$}}}
\put(0,80){\makebox(0,0){{\tiny $\bullet$}}}
\put(120,80){\makebox(0,0){{\tiny $\bullet$}}}
\put(-120,60){\makebox(0,0){{\tiny $\bullet$}}}
\put(-60,60){\makebox(0,0){{\tiny $\bullet$}}}
\put(0,60){\makebox(0,0){{\tiny $\bullet$}}}
\put(60,60){\makebox(0,0){{\tiny $\bullet$}}}
\put(120,60){\makebox(0,0){{\tiny $\bullet$}}}

\put(-120,40){\makebox(0,0){{\tiny $\bullet$}}}
\put(-90,40){\makebox(0,0){{\tiny $\bullet$}}}
\put(-60,40){\makebox(0,0){{\tiny $\bullet$}}}
\put(-30,40){\makebox(0,0){{\tiny $\bullet$}}}
\put(0,40){\makebox(0,0){{\tiny $\bullet$}}}
\put(30,40){\makebox(0,0){{\tiny $\bullet$}}}
\put(60,40){\makebox(0,0){{\tiny $\bullet$}}}
\put(90,40){\makebox(0,0){{\tiny $\bullet$}}}
\put(120,40){\makebox(0,0){{\tiny $\bullet$}}}

\put(-125,20){\makebox(0,0){{\small $\frac{0}{1}$}}}
\put(-105,27){\makebox(0,0){{\small $\frac{1}{5}$}}}
\put(-75,27){\makebox(0,0){{\small $\frac{2}{7}$}}}
\put(-45,27){\makebox(0,0){{\small $\frac{3}{8}$}}}
\put(-15,27){\makebox(0,0){{\small $\frac{3}{7}$}}}
\put(125,20){\makebox(0,0){{\small $\frac{1}{1}$}}}
\put(105,27){\makebox(0,0){{\small $\frac{4}{5}$}}}
\put(75,27){\makebox(0,0){{\small $\frac{5}{7}$}}}
\put(45,27){\makebox(0,0){{\small $\frac{5}{8}$}}}
\put(15,27){\makebox(0,0){{\small $\frac{4}{7}$}}}

\put(-120,20){\makebox(0,0){{\tiny $\bullet$}}}
\put(-105,20){\makebox(0,0){{\tiny $\bullet$}}}
\put(-90,20){\makebox(0,0){{\tiny $\bullet$}}}
\put(-75,20){\makebox(0,0){{\tiny $\bullet$}}}
\put(-60,20){\makebox(0,0){{\tiny $\bullet$}}}
\put(-45,20){\makebox(0,0){{\tiny $\bullet$}}}
\put(-30,20){\makebox(0,0){{\tiny $\bullet$}}}
\put(-15,20){\makebox(0,0){{\tiny $\bullet$}}}
\put(0,20){\makebox(0,0){{\tiny $\bullet$}}}
\put(120,20){\makebox(0,0){{\tiny $\bullet$}}}
\put(105,20){\makebox(0,0){{\tiny $\bullet$}}}
\put(90,20){\makebox(0,0){{\tiny $\bullet$}}}
\put(75,20){\makebox(0,0){{\tiny $\bullet$}}}
\put(60,20){\makebox(0,0){{\tiny $\bullet$}}}
\put(45,20){\makebox(0,0){{\tiny $\bullet$}}}
\put(30,20){\makebox(0,0){{\tiny $\bullet$}}}
\put(15,20){\makebox(0,0){{\tiny $\bullet$}}}

\put(-125,0){\makebox(0,0){{\small $\frac{0}{1}$}}}
\put(-112.5,-2){\makebox(0,0){{\small $\frac{1}{6}$}}}
\put(-97.5,-2){\makebox(0,0){{\small $\frac{2}{9}$}}}
\put(-82.5,-2){\makebox(0,0){{\small $\frac{3}{11}$}}}
\put(-67.5,-2){\makebox(0,0){{\small $\frac{3}{10}$}}}
\put(-52.5,-2){\makebox(0,0){{\small $\frac{4}{11}$}}}
\put(-37.5,-2){\makebox(0,0){{\small $\frac{5}{13}$}}}
\put(-22.5,-2){\makebox(0,0){{\small $\frac{5}{12}$}}}
\put(-7.5,-2){\makebox(0,0){{\small $\frac{4}{9}$}}}
\put(125,0){\makebox(0,0){{\small $\frac{1}{1}$}}}
\put(112.5,-2){\makebox(0,0){{\small $\frac{5}{6}$}}}
\put(97.5,-2){\makebox(0,0){{\small $\frac{7}{9}$}}}
\put(82.5,-2){\makebox(0,0){{\small $\frac{8}{11}$}}}
\put(67.5,-2){\makebox(0,0){{\small $\frac{7}{10}$}}}
\put(52.5,-2){\makebox(0,0){{\small $\frac{7}{11}$}}}
\put(37.5,-2){\makebox(0,0){{\small $\frac{8}{13}$}}}
\put(22.5,-2){\makebox(0,0){{\small $\frac{7}{12}$}}}
\put(7.5,-2){\makebox(0,0){{\small $\frac{5}{9}$}}}

\put(-120,5){\makebox(0,0){{\tiny $\bullet$}}}
\put(-112.5,5){\makebox(0,0){{\tiny $\bullet$}}}
\put(-105,5){\makebox(0,0){{\tiny $\bullet$}}}
\put(-97.5,5){\makebox(0,0){{\tiny $\bullet$}}}
\put(-90,5){\makebox(0,0){{\tiny $\bullet$}}}
\put(-82.5,5){\makebox(0,0){{\tiny $\bullet$}}}
\put(-75,5){\makebox(0,0){{\tiny $\bullet$}}}
\put(-67.5,5){\makebox(0,0){{\tiny $\bullet$}}}
\put(-60,5){\makebox(0,0){{\tiny $\bullet$}}}
\put(-52.5,5){\makebox(0,0){{\tiny $\bullet$}}}
\put(-45,5){\makebox(0,0){{\tiny $\bullet$}}}
\put(-37.5,5){\makebox(0,0){{\tiny $\bullet$}}}
\put(-30,5){\makebox(0,0){{\tiny $\bullet$}}}
\put(-22.5,5){\makebox(0,0){{\tiny $\bullet$}}}
\put(-15,5){\makebox(0,0){{\tiny $\bullet$}}}
\put(-7.5,5){\makebox(0,0){{\tiny $\bullet$}}}
\put(0,5){\makebox(0,0){{\tiny $\bullet$}}}
\put(120,5){\makebox(0,0){{\tiny $\bullet$}}}
\put(112.5,5){\makebox(0,0){{\tiny $\bullet$}}}
\put(105,5){\makebox(0,0){{\tiny $\bullet$}}}
\put(97.5,5){\makebox(0,0){{\tiny $\bullet$}}}
\put(90,5){\makebox(0,0){{\tiny $\bullet$}}}
\put(82.5,5){\makebox(0,0){{\tiny $\bullet$}}}
\put(75,5){\makebox(0,0){{\tiny $\bullet$}}}
\put(67.5,5){\makebox(0,0){{\tiny $\bullet$}}}
\put(60,5){\makebox(0,0){{\tiny $\bullet$}}}
\put(52.5,5){\makebox(0,0){{\tiny $\bullet$}}}
\put(45,5){\makebox(0,0){{\tiny $\bullet$}}}
\put(37.5,5){\makebox(0,0){{\tiny $\bullet$}}}
\put(30,5){\makebox(0,0){{\tiny $\bullet$}}}
\put(22.5,5){\makebox(0,0){{\tiny $\bullet$}}}
\put(15,5){\makebox(0,0){{\tiny $\bullet$}}}
\put(7.5,5){\makebox(0,0){{\tiny $\bullet$}}}

\end{picture}
\end{center}
\caption{The sum-level sets as Stern-Brocot intervals}
\label{Fig:2}
\end{figure}
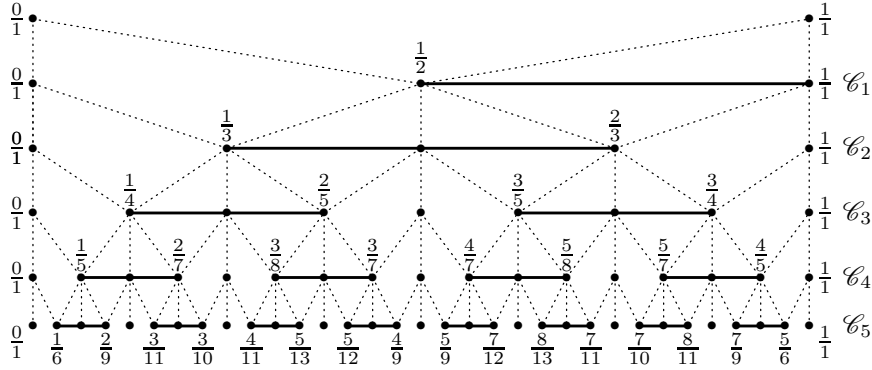

\begin{remark}\label{R1}
The sets $\mathscr{C}_n$ can be viewed in terms of the Stern-Brocot sequence $\mathcal{S}_n:=\{s_{n,k}/t_{n,k}:k=1,\ldots,2^n+1\}$ defined recursively by the following:
\begin{itemize}
\item $s_{0,1}:=0$ and $s_{0,2}:=t_{0,1}:=t_{0,2}:=1$;
\item $s_{n+1,2k-1}:=s_{n,k}$ and $t_{n+1,2k-1}:=t_{n,k}$ for $k=1,\ldots,2^n+1$;
\item $s_{n+1,2k}:=s_{n,k}+s_{n,k+1}$ and $t_{n+1,2k}:=t_{n,k}+t_{n,k+1}$ for $k=1,\ldots,2^n+1$.
\end{itemize}
Specifically, we have $\mathscr{C}_1=[1/2,1]=[s_{1,2}/t_{1,2},s_{1,3}/t_{1,3}]$, and for $n\geq2$,
 \[\mathscr{C}_n=\bigcup_{k=1}^{2^{n-2}}\left[\frac{s_{n,4k-2}}{t_{n,4k-2}},\frac{s_{n,4k}}{t_{n,4k}}\right].\]
This was the characterization of $\mathscr{C}_n$ considered by Fiala and Kleban \cite{FK} motivated by their study of spin chain models. The relationship between the Farey map and the Stern-Brocot sequence was exploited in \cite{KSt1} to prove the equidistribution of certain weighted subsets of the Stern-Brocot sequence.
\end{remark}

\section{Proof of Theorem \ref{T2}}

\subsection{Proof of \eqref{e04}}
The case in which $u\geq1/2$ follows from the proof of \cite[Theorem 1.1]{KSt2}, so assume $u\in(0,1/2)$. Let $\varphi_0:[0,1]\rar[0,1]$ be defined by $\varphi_0(x):=x$ (note $d\lambda=\varphi_0\,d\mu$) so that
\[\lambda(\mathscr{C}_n^u)=\int_u^1\hat{F}^{n-1}\varphi_0\,d\mu.\]
By \cite[Lemma 3.2]{KSl1}, $\hat{F}$ maps the set of functions $\{f\in C^2(0,1):f'>0,f''\leq0\}$ into itself, and thus it suffices to show that 
\[\int_u^1\hat{F}f\,d\mu<\int_u^1f\,d\mu\]
whenever $f\in L^1(\mu)$ is increasing. This follows from
\begin{align*}
\int_u^1f\,d\mu-\int_u^1\hat{F}f\,d\mu&=\int_u^1f\,d\mu-\int_{F^{-1}[u,1]}f\,d\mu=\int_u^1f\,d\mu-\int_{u/(1+u)}^{1/(1+u)}f\,d\mu\\
&=\int_{1/(1+u)}^1f\,d\mu-\int_{u/(1+u)}^uf\,d\mu\geq\left(f\left(\frac{1}{1+u}\right)-f(u)\right)\log(1+u)>0.
\end{align*}

\subsection{Proof of \eqref{e05}} We first consider the case where $u=1/N$, with $N\in\N$ and $N\geq2$. Define the function $a:\R\rar\R$ by
\[a(\sigma):=\frac{1}{\log N}\sum_{k=0}^{\lfloor\sigma\rfloor}\lambda(\mathscr{C}_{k+1}^u),\]
which is the $\mu$-average of the function $\hat{F}_\sigma\varphi_0:=\sum_{k=0}^{\lfloor\sigma\rfloor}\hat{F}^k\varphi_0$ on $\mathscr{C}_1^u=[1/N,1]$ by
\[\frac{1}{\mu(\mathscr{C}_1^u)}\int_{\mathscr{C}_1^u}\sum_{k=0}^{\lfloor\sigma\rfloor}\hat{F}^{k}\varphi_0\,d\mu=\frac{1}{\log N}\sum_{k=0}^{\lfloor\sigma\rfloor}\int_{F^{-k}(\mathscr{C}_1^u)}\varphi_0\,d\mu=\frac{1}{\log N}\sum_{k=0}^{\lfloor\sigma\rfloor}\lambda(\mathscr{C}_{k+1}^u).\]
We have the following bound on the difference of $\hat{F}_\sigma\varphi_0$ and $a(\sigma)$.

\begin{lemma}\label{L1}
For all $\sigma\in\R$ and $x\in\mathscr{C}_1^u$,
\[|\hat{F}_\sigma\varphi_0(x)-a(\sigma)|\leq\frac{N(N-1)}{2}.\]
\end{lemma}

\begin{proof}
Without loss of generality, we assume that $\sigma=n\in\N\cup\{0\}$.  By \cite[Lemma 3.2]{KSl1}, we know that $\hat{F}_n\varphi_0$ is increasing. So the difference between $\hat{F}_n\varphi_0(x)$ and $a(n)$ is at most $\hat{F}_n\varphi_0(1)-\hat{F}_n\varphi_0(1/N)$. Using the equality
\[\hat{F}^k\varphi_0\left(\frac{1}{j}\right)=\frac{j}{j-1}\hat{F}^{k+1}\varphi_0\left(\frac{1}{j-1}\right)-\frac{1}{j-1}\hat{F}^k\varphi_0\left(\frac{j-1}{j}\right)\qquad(j\in\N,j\geq2)\]
following from \eqref{e06}, and the fact that $\hat{F}^k\varphi_0$ is increasing for each $k\in\N\cup\{0\}$, we have
\begin{align*}
\hat{F}_n\varphi_0(1)-\hat{F}_n\varphi_0\left(\frac{1}{N}\right)&=\sum_{k=0}^n\left(\hat{F}^{k}\varphi_0(1)-\hat{F}^{k}\varphi_0\left(\frac{1}{N}\right)\right)\\
&=\sum_{k=0}^n\left(\hat{F}^{k}\varphi_0(1)-\frac{N}{N-1}\hat{F}^{k+1}\varphi_0\left(\frac{1}{N-1}\right)+\frac{1}{N-1}\hat{F}^k\varphi_0\left(\frac{N-1}{N}\right)\right)\\
&\leq\frac{N}{N-1}\sum_{k=0}^n\left(\hat{F}^k\varphi_0(1)-\hat{F}^{k+1}\varphi_0\left(\frac{1}{N-1}\right)\right).
\end{align*}
Using this inequality recursively, and also the equality $\hat{F}f(1)=f(1/2)$, yields
\begin{align*}
\hat{F}_n\varphi_0(1)-\hat{F}_n\varphi_0\left(\frac{1}{N}\right)&\leq\frac{N}{2}\sum_{k=0}^n\left(\hat{F}^k\varphi_0(1)-\hat{F}^{k+N-2}\varphi_0\left(\frac{1}{2}\right)\right)\\
&=\frac{N}{2}\sum_{k=0}^n(\hat{F}^k\varphi_0(1)-\hat{F}^{k+N-1}\varphi_0(1))\leq\frac{N(N-1)}{2}.
\end{align*}
\end{proof}

Next, we let $S:(0,\infty)\rar\R$ be the Laplace transform of $a$ given by
\[S(\sigma):=\int_{0-}^\infty e^{-t/\sigma}\,da(t)=\frac{1}{\log N}\sum_{n=0}^\infty e^{-n/\sigma}\lambda(\mathscr{C}_{n+1}^u)\]
and prove the following bound similar to Lemma \ref{L1}.

\begin{lemma}\label{L2}
For all $x\in\mathscr{C}_1^u$ and all $\sigma>0$,
\[\left|\sum_{n=0}^\infty e^{-n/\sigma}\hat{F}^n\varphi_0(x)-S(\sigma)\right|\leq\frac{N(N-1)}{2}.\]
\end{lemma}

\begin{proof}
We first note the equality
\begin{equation}\label{e07}
\sum_{n=0}^\infty a_ne^{-n/\sigma}=(1-e^{-1/\sigma})\sum_{n=0}^\infty e^{-n/\sigma}\left(\sum_{k=0}^na_k\right),
\end{equation}
which holds for all sequences $(a_n)$ satisfying $\sum_{k=0}^na_k=O(n)$ as $n\rar\infty$ and all $\sigma>0$.

Let $x\in\mathscr{C}_1^u$, $\delta_n(x):=\hat{F}_n\varphi_0(x)-a(n)$, and $\sigma>0$. Using \eqref{e07} twice, we have
\begin{align*}
\sum_{n=0}^\infty e^{-n/\sigma}\hat{F}^n\varphi_0(x)&=(1-e^{-1/\sigma})\sum_{n=0}^\infty e^{-n/\sigma}\hat{F}_n\varphi_0(x)=(1-e^{-1/\sigma})\sum_{n=0}^\infty e^{-n/\sigma}(a(n)+\delta_n(x))\\
&=S(\sigma)+(1-e^{-1/\sigma})\sum_{n=0}^\infty e^{-n/\sigma}\delta_n(x).
\end{align*}
Since $|\delta_n(x)|\leq N(N-1)/2$ for all $n\geq0$, we have
\[\left|(1-e^{-1/\sigma})\sum_{n=0}^\infty e^{-n/\sigma}\delta_n(x)\right|\leq(1-e^{-1/\sigma})\sum_{n=0}^\infty e^{-n/\sigma}\frac{N(N-1)}{2}=\frac{N(N-1)}{2}.\]
\end{proof}

To continue the proof, we will make use of the following equality given by \cite[Lemma 3.8.4]{A1}.
\begin{equation}\label{e08}
\int_A\left(\sum_{n=0}^\infty e^{-n/\sigma}\hat{F}^nf\right)(1-e^{-\varphi_A/\sigma})\,d\mu=\sum_{n=0}^\infty e^{-n/\sigma}\int_{A_n}f\,d\mu
\end{equation}
Here, $f$ is any function in $L^1(\mu)$, $\sigma$ is any positive real number, $A\subseteq[0,1]$ is any subset such that $\mu(A)<\infty$, $A_0:=A$, and $A_n:=F^{-n}A\bsl\bigcup_{k=0}^{n-1}F^{-k}A$ for $n\geq1$. Also, $\varphi_A:A\rar\N$ is the return time function on $A$ defined by $\varphi_A(x):=\min\{n\in\N:F^n(x)\in A\}$.

Letting $A=\mathscr{C}_1^u$ and $f=\varphi_0$ in \eqref{e08}, and noting that $A_n=[\frac{1}{n+N},\frac{1}{n+N-1})$ for $n\geq1$, we have
\begin{align*}
\int_{1/N}^1\left(\sum_{n=0}^\infty e^{-n/\sigma}\hat{F}^n\varphi_0\right)(1-e^{-\varphi_A/\sigma})\,d\mu&=\frac{N-1}{N}+\sum_{n=1}^\infty e^{-n/\sigma}\int_{1/(n+N)}^{1/(n+N-1)}\varphi_0\,d\mu\\
&=\frac{N-1}{N}+\sum_{n=1}^\infty\frac{e^{-n/\sigma}}{(n+N)(n+N-1)}.
\end{align*}
On the other hand, using Lemma \ref{L2}, we see that the left side of the above is also equal to
\begin{align*}
(S(\sigma)+O_N(1))(1-e^{-1/\sigma})&\int_{1/N}^1\left(\sum_{n=0}^\infty e^{-n/\sigma}\hat{F}^n1\right)(1-e^{-\varphi_A/\sigma})\,d\mu\\
&=(S(\sigma)+O_N(1))(1-e^{-1/\sigma})\left(\mu(\mathscr{C}_1^u)+\sum_{n=1}^\infty e^{-n/\sigma}\int_{1/(n+N)}^{1/(n+N-1)}\,d\mu\right)\\
&=(S(\sigma)+O_N(1))(1-e^{-1/\sigma})\left(\log N+\sum_{n=1}^\infty e^{-n/\sigma}\log\left(\frac{n+N}{n+N-1}\right)\right)
\end{align*}
as $\sigma\rar\infty$. (Note that \eqref{e08} holds for the constant function $f=1$ in spite of the fact that $1\notin L^1(\mu)$ since $\sum_{n=0}^\infty e^{-n/\sigma}\hat{F}^n1$ has finite integral over $\mathscr{C}_1^u$.) For our next step, we determine the asymptotic behavior of
\[\frac{N-1}{N}+\sum_{n=1}^\infty\frac{e^{-n/\sigma}}{(n+N)(n+N-1)}\quad\mbox{and}\quad\log N+\sum_{n=1}^\infty e^{-n/\sigma}\log\left(\frac{n+N}{n+N-1}\right).\]

\begin{lemma}\label{L3}
\[\frac{N-1}{N}+\sum_{n=1}^\infty\frac{e^{-n/\sigma}}{(n+N)(n+N-1)} = 1 + O\left(\frac{\log\sigma}{\sigma}\right).\qquad(\sigma\rar\infty)\]
\end{lemma}

\begin{proof}
Let $S_1:\R\rar\R$ be defined by \[S_1(t):=\frac{N-1}{N}1_{[0,\infty)}(t)+\sum_{n=1}^{\lfloor t\rfloor}\frac{1}{(n+N)(n+N-1)}
=\begin{cases}1-\frac{1}{\lfloor t\rfloor+N}&\mbox{if }t\geq0\\
0&\mbox{if }t<0.\end{cases}\]
Then for $\sigma>0$,
\begin{align*}
\frac{N-1}{N}+\sum_{n=1}^\infty\frac{e^{-n/\sigma}}{(n+N)(n+N-1)}&=\int_{0-}^\infty e^{-t/\sigma}\,dS_1(t)=\frac{1}{\sigma}\int_0^\infty\left(1-\frac{1}{\lfloor t\rfloor+N}\right)e^{-t/\sigma}\,dt\\
&=1-\int_0^\infty\frac{e^{-x}\,dx}{\lfloor \sigma x\rfloor+N}.
\end{align*}
Since the inequality $\lfloor t\rfloor +N\geq\frac{1}{2}(t+2)$ holds for $t\geq0$, we have
\begin{align*}
\int_0^\infty\frac{e^{-x}\,dx}{\lfloor \sigma x\rfloor+N}\leq2\int_0^\infty\frac{e^{-x}\,dx}{\sigma x+2}\leq\int_0^1\frac{2\,dx}{\sigma x+2}+\int_1^\infty\frac{2e^{-x}}{\sigma}\,dx\ll\frac{\log\sigma}{\sigma}. \qquad(\sigma\rar\infty)
\end{align*}
\end{proof}

\begin{lemma}\label{L4} We have
\[\log N+\sum_{n=1}^\infty e^{-n/\sigma}\log\left(\frac{n+N}{n+N-1}\right)=\log(\sigma+N)+C+O_N\left(\frac{\log\sigma}{\sigma}\right),\qquad(\sigma\rar\infty)\]
where
\[C:=\int_0^1\frac{e^{-x}-1}{x}\,dx+\int_1^\infty\frac{e^{-x}}{x}\,dx.\]
\end{lemma}

\begin{proof}
Let $S_2:\R\rar\R$ be defined by
\[S_2(t):=(\log N)1_{[0,\infty)}(t)+\sum_{n=1}^{\lfloor t\rfloor}\log\left(\frac{n+N}{n+N-1}\right)
=\begin{cases} \log(\lfloor t\rfloor+N)&\mbox{if }t\geq0\\
0&\mbox{if }t<0.\end{cases}\]
Then for $\sigma>0$,
\begin{align*}
\log N+\sum_{n=1}^\infty e^{-n/\sigma}\log\left(\frac{n+N}{n+N-1}\right)&=\int_{0-}^\infty e^{-t/\sigma}\,dS_2(t)=\frac{1}{\sigma}\int_0^\infty e^{-t/\sigma}\log(\lfloor t\rfloor+N)\,dt\\
&=\int_0^\infty e^{-x}\log(\lfloor\sigma x\rfloor+N)\,dx\\
&=\int_0^\infty e^{-x}\log(\sigma x+N)\,dx-\int_0^\infty e^{-x}\log\left(\frac{\sigma x+N}{\lfloor\sigma x\rfloor+N}\right)\,dx.
\end{align*}
Using the inequality $\log(1+x)\leq x$, we have
\[\int_0^\infty e^{-x}\log\left(\frac{\sigma x+N}{\lfloor\sigma x\rfloor+N}\right)\,dx=\int_0^\infty e^{-x}\log\left(1+\frac{\{\sigma x\}}{\lfloor\sigma x\rfloor+N}\right)\,dx\ll\int_0^\infty\frac{e^{-x}\,dx}{\lfloor\sigma x\rfloor+N},\]
which is $O(\sigma^{-1}\log\sigma)$ as $\sigma\rar\infty$ by the proof of Lemma \ref{L3}.

Next, integration by parts yields
\begin{equation}\label{e09}
\int_0^\infty e^{-x}\log(\sigma x+N)\,dx=\log N+\int_0^\infty\frac{\sigma e^{-x}\,dx}{\sigma x+N}.
\end{equation}
To continue, we consider the integral on the right over $[0,1]$ by writing
\[\int_0^1\frac{\sigma e^{-x}\,dx}{\sigma x+N}=\int_0^1\frac{\sigma\,dx}{\sigma x+N}+\int_0^1\frac{\sigma(e^{-x}-1)}{\sigma x+N}\,dx.\]
The first integral on the right equals $\log(\sigma+N)-\log N$, while the second equals
\begin{align*}
\int_0^1\frac{e^{-x}-1}{x}\,dx-N\int_0^1\frac{e^{-x}-1}{x(\sigma x+N)}\,dx&=\int_0^1\frac{e^{-x}-1}{x}\,dx+O\left(\int_0^1\frac{N\,dx}{\sigma x+N}\right)\\
&=\int_0^1\frac{e^{-x}-1}{x}\,dx+O_N\left(\frac{\log\sigma}{\sigma}\right).\qquad(\sigma\rar\infty)
\end{align*}
Now considering the integral in \eqref{e09} over $[1,\infty)$, we write
\begin{align*}
\int_1^\infty\frac{\sigma e^{-x}\,dx}{\sigma x+N}=\int_1^\infty\frac{e^{-x}}{x}\,dx-N\int_1^\infty\frac{e^{-x}\,dx}{x(\sigma x+N)}=\int_1^\infty\frac{e^{-x}}{x}\,dx+O_N\left(\frac{1}{\sigma}\right).
\end{align*}
Putting these results together proves the lemma.
\end{proof}

Lemmas \ref{L3} and \ref{L4} and the equalities preceding them gives
\[(S(\sigma)+O_N(1))(1-e^{-1/\sigma})\left(\log(\sigma+N)+C+O_N\left(\frac{\log\sigma}{\sigma}\right)\right)=1+O\left(\frac{\log\sigma}{\sigma}\right),\qquad(\sigma\rar\infty)\]
and as a result,
\begin{equation}\label{e10}
S(\sigma)=\frac{\sigma}{\log\sigma+C}+O_N(1).\qquad(\sigma\rar\infty)
\end{equation}
At this point, an application of Karamata's Tauberian theorem \cite{K} then yields
\[\sum_{k=0}^n\lambda(\mathscr{C}_{k+1}^u)\sim\frac{n}{\log_Nn}.\qquad(n\rar\infty)\]
Furthermore, one can apply an adaptation of Freud's effective version of Karamata's theorem \cite{F} (see also \cite[Section 7.4]{T}) accommodating logarithms to \eqref{e10} in order to prove
\[\sum_{k=0}^n\lambda(\mathscr{C}_{k+1}^u)=\frac{n}{\log_Nn}\left(1+O_N\left(\frac{1}{\log\log n}\right)\right).\qquad(n\rar\infty)\]
To obtain an error term of $O_N(1/\log n)$, we evaluate the equality
\begin{align*}
&(S(\sigma)+O_N(1))(1-e^{-1/\sigma})\left(\log N+\sum_{n=1}^\infty e^{-n/\sigma}\log\left(\frac{n+N}{n+N-1}\right)\right)=1+O\left(\frac{\log\sigma}{\sigma}\right)
\end{align*}
more precisely. Instead of directly establishing an asymptotic equality for $S(\sigma)$, we divide by $1-e^{-1/\sigma}$ and multiply the series expression for $S(\sigma)$ together with the other series on the left side. Together with Lemma \ref{L4}, this process yields
\[\frac{1}{\log N}\sum_{n=0}^\infty\left(\sum_{k=0}^n\lambda(\mathscr{C}_{k+1}^u)\ell_N(n-k)\right)e^{-n/\sigma}=\sigma\left(1+O_N\left(\frac{\log\sigma}{\sigma}\right)\right),\qquad(\sigma\rar\infty)\]
where $\ell_N(0):=\log N$ and $\ell_N(n):=\log\left(\frac{n+N}{n+N-1}\right)$ for $n>0$. Now a direct application of Freud's effective Tauberian theorem yields
\[\sum_{k=0}^n\sum_{j=0}^k\lambda(\mathscr{C}_{j+1}^u)\ell_N(k-j)=n\log N\left(1+O_N\left(\frac{1}{\log n}\right)\right).\qquad(n\rar\infty)\]
The left side of this expression is equal to
\begin{align*}
\sum_{k=0}^n\left(\lambda(\mathscr{C}_{k+1}^u)\log N+\sum_{j=0}^{k-1}\lambda(\mathscr{C}_{j+1}^u)\ell_N(k-j)\right)&=\log N\sum_{k=0}^n\lambda(\mathscr{C}_{k+1}^u)+\sum_{j=0}^{n-1}\lambda(\mathscr{C}_{j+1}^u)\sum_{k=j+1}^n\ell_N(k-j)\\
&=\log N\sum_{k=0}^n\lambda(\mathscr{C}_{k+1}^u)+\sum_{j=0}^{n-1}\lambda(\mathscr{C}_{j+1}^u)\log\left(\frac{n-j+N}{N}\right)\\
&=\sum_{k=0}^n\lambda(\mathscr{C}_{k+1}^u)\log(n-k+N),
\end{align*}
where the second equality follows from the definition of $\ell_N$ and telescoping. We can rewrite the last expression above as
\[\log(n+N)\sum_{k=0}^n\lambda(\mathscr{C}_{k+1}^u)+\sum_{k=1}^n\lambda(\mathscr{C}_{k+1}^u)\log\left(1-\frac{k}{n+N}\right).\]
So if we can show that
\begin{equation}\label{e11}
\sum_{k=1}^n\lambda(\mathscr{C}_{k+1}^u)\log\left(1-\frac{k}{n+N}\right)=O\left(\frac{n}{\log n}\right),\qquad(n\rar\infty)
\end{equation}
then we have
\[\sum_{k=0}^n\lambda(\mathscr{C}_{k+1}^u)=\frac{n}{\log_Nn}\left(1+O_N\left(\frac{1}{\log n}\right)\right).\]

Since individual terms in the left side of \eqref{e11} decay to $0$ as $n\rar\infty$, we can consider the sum starting from $k=3$. We have
\begin{align*}
\left|\sum_{k=3}^n\lambda(\mathscr{C}_{k+1}^u)\log\left(1-\frac{k}{n+N}\right)\right|&=\sum_{k=3}^n\lambda(\mathscr{C}_{k+1}^u)\sum_{j=1}^\infty\frac{1}{j}\left(\frac{k}{n+N}\right)^j=\sum_{j=1}^\infty\frac{1}{j(n+N)^j}\sum_{k=3}^nk^j\lambda(\mathscr{C}_{k+1}^u)\\
&\ll\sum_{j=1}^\infty\frac{1}{j(n+N)^j}\sum_{k=3}^n\frac{k^j}{\log k}\ll\sum_{j=1}^\infty\frac{1}{j(n+N)^j}\int_3^{n+1}\frac{x^j\,dx}{\log x}\\
&\ll\sum_{j=1}^\infty\frac{1}{j(n+N)^j}\left(\frac{(n+1)^{j+1}}{(j+1)\log(n+1)}\right)\\
&\ll\frac{n}{\log n}\sum_{j=1}^\infty\frac{1}{j(j+1)}\left(\frac{n+1}{n+N}\right)^j\ll\frac{n}{\log n}.\qquad(n\rar\infty)
\end{align*}
This proves Theorem \ref{T2} in the case that $u=1/N$.

For the general case $u\in(0,1)$, let $N=\lceil 1/u\rceil$ so that $[u,1]\subseteq[1/N,1]$. Then for $x\in[1/N,1]$, we have
\[\sum_{k=0}^n\hat{F}^k\varphi_0(x)=\frac{1}{\log N}\sum_{k=0}^n\lambda(\mathscr{C}_{k+1}^{1/N})+O_N(1)=\frac{n}{\log n}\left(1+O_N\left(\frac{1}{\log n}\right)\right).\qquad(n\rar\infty)\]
Integrating the first and last expressions over $[u,1]$ yields
\[\sum_{k=0}^n\lambda(\mathscr{C}_{k+1}^u)=\frac{n\log(1/u)}{\log n}\left(1+O_N\left(\frac{1}{\log n}\right)\right),\qquad(n\rar\infty)\]
completing the proof.

\begin{remark}\label{R2}
Theorem \ref{T1} can be extended to obtain an effective version of \cite[Proposition 1.3]{KSt1}. Specifically, if we let $f\in C^2[0,1]$ and $\|f\|_{C^2}:=\|f\|_\infty+\|f'\|_\infty+\|f''\|_\infty$, where $\|g\|_\infty:=\sup\{|g(x)|:x\in[0,1]\}$, we have
\begin{equation}\label{e12}
\int_{F^{-(n-1)}[\alpha,\beta]}f\,d\lambda=\frac{\log(\beta/\alpha)}{\log n}\left(\int_0^1f\,d\lambda+O_{\alpha,\beta}\left(\frac{\|f\|_{C^2}}{\log n}\right)\right).\qquad(n\rar\infty)
\end{equation}
The deduction of this involves essentially repeating the proof of Theorem \ref{T2} with $\varphi_0$ replaced by a general function $\varphi\in L^1(\mu)\cap C^2[0,1]$ satisfying $\varphi,\varphi'\geq0$ and $\varphi''\leq0$, and for a given $f\in C^2[0,1]$, writing the function $x\mapsto xf(x)$ as the difference of two appropriate functions with the properties of $\varphi$. One can then use \eqref{e12} to obtain an effective version of the Stern-Brocot equidistribution result \cite[Theorem 1.2]{KSt1}. Full details can be seen in the author's future Ph.D.~thesis.
\end{remark}

\ 

\noindent \textbf{Acknowledgements.}  I thank my advisor Florin Boca for suggesting this research topic and his guidance and encouragement. I thank Harold Diamond for his suggestions, which helped me improve the error term in Theorem \ref{T1}. I also thank Dalia Terhesiu and the referee for their helpful comments and suggestions. I also acknowledge support from National Science Foundation Grant DMS 08-38434 \lq\lq EMSW21-MCTP: Research Experience for Graduate Students.\rq\rq


\begin{thebibliography}{99}

\bibitem{A1} J.~Aaronson. \textit{An introduction to infinite ergodic theory}, Mathematical Surveys and Monographs 50, A.M.S., Providence, RI, 1997.

\bibitem{A2} J.~Aaronson. \textit{Random $f$-expansions}, Ann.~Probab. 14 (1986), no.~3, 1037--1057.

\bibitem{FK} J.~Fiala, P.~Kleban. \textit{Intervals between Farey fractions in the limit of infinite level}, Ann.~Sci.~Math.~Qu\'ebec 34 (2010), no.~1, 63--71.

\bibitem{F} G.~Freud. \textit{Restglied eines Tauberschen Satzes I}, Acta Math.~Acad.~Sci.~Hung. 2 (1951), 299--308.

\bibitem{IK} M.~Iosifescu, C.~Kraaikamp. \textit{Metrical theory of continued fractions}, Mathematics and Its Applications 547, Kluwer Academic Publishers, Dordrecht, 2002.

\bibitem{I1} S.~Isola, \textit{From infinite ergodic theory to number theory (and possibly back)}, Chaos, Solitons, and Fractals 44 (2011), no.~7, 467--479.

\bibitem{I2} S.~Isola, \textit{On the spectrum of Farey and Gauss maps}, Nonlinearity 15 (2002), no.~5, 1521--1539.

\bibitem{K} J.~Karamata. \textit{Neuer Beweis und Verallgemeinerung der Tauberschen S\"atze, welche die Laplacesche und Stieltjessche Transformation betreffen}, J.~Reine Angew.~Math. 164 (1931) 27--39.

\bibitem{KSl1} M.~Kesseb\"ohmer, M.~Slassi. \textit{A distributional limit law for the continued fraction digit sum}, Math.~Nachr. 281 (2008), no.~9, 1294--1306.

\bibitem{KSl2} M.~Kesseb\"ohmer, M.~Slassi. \textit{Limit laws for distorted critical return time processes in infinite ergodic theory}, Stoch.~Dyn. 7 (2007), no.~1, 103--121.

\bibitem{KSt1} M.~Kesseb\"ohmer, B.~Stratmann. \textit{A dichotomy between uniform distributions of the Stern-Brocot and the Farey sequence}, Uniform Distribution Theory 7 (2012), no.~2, 21--33.

\bibitem{KSt2} M.~Kesseb\"ohmer, B.~Stratmann. \textit{On the asymptotic behavior of the Lebesgue measure of sum-level sets for continued fractions}, Discrete Contin.~Dyn.~Sys. 32 (2012), no.~7, 2437--2451.

\bibitem{MT} I.~Melbourne, D.~Terhesiu. \textit{Operator renewal theory and mixing rates for dynamical systems with infinite measure}, Invent.~Math. 189 (2012), no.~1, 61--110.

\bibitem{P} T.~Prellberg, \textit{Towards a complete determination of the spectrum of a transfer operator associated with intermittency}, J.~Phys.~A 36 (2003), no.~10, 2455--1461.

\bibitem{T} G.~Tenenbaum, \textit{Introduction to Analytic and Probabilistic Number Theory}, Cambridge Studies in Advanced Mathematics 46, Cambridge University Press, Cambridge, 1995.


\end{thebibliography}
\end{document}